\documentclass{article}
\usepackage{graphicx}
\usepackage{amsmath, amsthm, amssymb}
\usepackage{amsfonts}

\newtheorem{theorem}{Theorem}[section]

\newtheorem{property}[theorem]{Property}

\newcommand{\no}{\noindent}
\newcommand{\su}{\textsuperscript}

\begin{document}
\title{Reversibility Checking for Markov Chains}
\author{Q. Jiang\su{a}, M. Hlynka\su{a},  P.H. Brill\su{a,b} and C.H. Cheung\su{a}\\
a. Department of Mathematics and Statistics\\
b. Department of Management Science\\
University of Windsor\\
Windsor, Ontario, Canada N9B 3P4\\
Corresponding Author: hlynka@uwindsor.ca}
\date{June 1, 2018}
\maketitle
\no
Keywords: Reversible Markov chain, detailed balance equations, Kolmogorov criterion  \\
Mathematics Subject Classification: 60J10, 60J22\\

\begin{abstract}
In this paper, we present reversibility preserving operations on Markov chain transition matrices. Simple row and column operations allow us to create new reversible transition matrices and yield  an easy  method for checking a Markov chain for reversibility.   
\end{abstract}

\newpage
\section{Introduction}
\large
Reversible Markov chains show up in many diverse areas. For example, they occur in MCMC (Markov Chain Monte Carlo) analyses (see \cite{Ald} Aldous and Fill, 2001). They have geological applications as in \cite{Ric} Richman and Sharp, 1991. They have applications in genetics models and queueing networks (see \cite{Kel} Kelly, 1978). McCullagh \cite{McC} (1982) and Sharp and Markham \cite{Sha} (2000) look at quasi symmetry and reversibility. More recent work has been done by  Pistone and Rogantin \cite {Pis} (2013). \\

\subsection{Notation}
We use standard Markov chain notation, as in \cite{Dur} Durrett, 2012.\\ 
  Let $P=[p_{ij}]$ be the probability transition matrix for an ergodic Markov chain $X(t), t=0,1,2,\dots$ with states 
$\{1,2,\dots,n\}$ where $p_{ij}=P(X(t)=j|X(t-1)=i)$ for $t=1,2\dots$. 
 Let $P^{(k)}=[p_{ij}^{(k)}]$ be the $k$ step transition matrix.  For an ergodic Markov chain, the limiting probability and stationary row vector $\vec{\pi}=(\pi_1,\pi_2,\dots)$ with 
$\pi_j = \lim_{k \rightarrow \infty} p_{ij}^{(k)}$
   exists and is independent of $i$. The limiting vector $\vec{\pi}$ can be found by solving the balance equations
   \begin{align}
      \vec{\pi} &= \vec{\pi} P,   \end{align}
   subject to $\displaystyle \sum_{j = 1}^n \pi_j = 1$. 

\subsection{Reversible Process}

From Kelly \cite{Kel}, (1978), an ergodic Markov chain ${X(t)}$ on state space $S$ is reversible if ($X(0), X(1), \dots, X(t)$) has the 
   same distribution as $(X(t), X(t-1), \dots, X(0))$ for all $t$. Given $P$ and $\underline{\pi}$, the chain $X(t)$ is reversible iff for all $i,j$, it satisfies the detailed balance equations
   \begin{align}
      \pi_ip_{ij} = \pi_jp_{ji}, \label{dbe}
   \end{align}

\subsection{Kolmogorov's Check for Reversibility}

   It may be desirable to verify reversibility before solving for the stationary vector $\vec{\pi}$, since reversibility allows a simple method for finding $\vec{\pi}$. We can also be interested in reversibility for other reasons. We refer to a transition matrix as being reversible if the corresponding Markov chain is reversible.
 
   One method is to use Kolmogorov's loop criterion  (see \cite{Kel} Kelly, 1978, Chapter 1). An ergodic Markov chain is reversible if and only if    
\begin{align}
      p_{j_0j_1} \ p_{j_1j_2} \ \dots \ p_{j_{k-1}j_k} \ p_{j_kj_0} = p_{j_0j_k} \ p_{j_kj_{k-1}} \ \dots \ p_{j_2j_1} \ p_{j_1,j_0}, \label{kl}
   \end{align}
   for every finite sequence of distinct states $j_0, j_1, j_2, \dots, j_k $ 

   Some matrices can be easily checked for reversibility by Kolmogorov's loop criterion.  For a two-state Markov chain $X(t)$, Kolmogorov's criterion is always satisfied since    $ p_{12}p_{21} = p_{12}p_{21}$. Also, if the transition matrix $P$ is symmetric, then $p_{ij} = p_{ji}$ for all $i,j$, so Kolmogorov's criterion is satisfied and the chain is reversible. Usually loop checking involves much computational work. 

In words, Kolmogorov's loop criterion says that a Markov transition matrix is reversible iff for every loop of distinct states, the forward loop probability product equals the backward loop probability product.  

One difficulty with Kolmogorov's method is that the number of loops that need to be checked grows very quickly with $n$ where $n$ is the number of states. We first present a result about the number of equations that must be checked in
order to apply Kolmogorov's criterion. 

Kelly \cite{Kel} notes (Exercise 1.5.2) that if there is a state which can be accessed from every other state in exactly 1 step (i.e. a column of the transition matrix with no zero entries), then it is sufficient to check loops of only three states. However, it is possible that no such state exists.
 
Another technique to check for reversibility is presented in  Richman and Sharp, 1991(\cite{Ric}). In their paper, they basically suggest premultiplying the probability transition matrix $P$ by a diagonal matrix $D$ formed by ratios of the entries in a particular nonzero row and its corresponding column. One difficulty with their method is that there may not exist such a nonzero row, further their result is stated in terms of tally matrices rather than probability transition matrices.  

In this paper, we present a new method (that allows zero entries to appear).  This new method is convenient and uses only traditional matrix row and column operations. 

\section {Counting Kolmogorov Loops}

The work in this section appears in Jiang \cite{Jia} (2011)
 \begin{property}
   For an $n$ state Markov chain, with $n\geq 3$, the number of equations that must be checked for reversibility by Kolmogorov's method is 
   \begin{align}
      \displaystyle \sum_{i=3}^n \binom{n}{i} \frac{(i-1)!}{2}.
   \end{align}
   \end{property}
   \textbf{Proof:}
     For a three-state Markov chain, we note that only one equation
   \begin{align}
      \nonumber p_{12}p_{23}p_{31} &= p_{13}p_{32}p_{21}
   \end{align}
is needed since no length 2 loops need to be checked and any other length 3 loop with the same states results in the same equation. For $n = 4$, we must check each loop of 3 states and each loop of 4 states. For three state loops paths, we choose any 3 out of 4 states and there is only one equation for each. For the four-state loops, we fix the starting state. Then there are $3!$ orders  for the other states. However, since the other side of equation is just the reversed path, there are only $\frac{3!}{2}$ paths involving four states, with the first state fixed. In total, we need to check 
   \begin{align}
      \binom{4}{3} + \binom{4}{4}\frac{(4-1)!}{2} = 7
   \end{align}
equations.  A similar argument holds for larger values of $n$. \qed \\

\begin{table*}[h]
	\centering 
		\begin{tabular}{|r|r|r|r|} \hline 
			 $n$ & Number of equations & $n$ & Number of equations  \\ \hline
			1 & 0 & 6 & 197\\
			2 & 0 & 7 & 1172\\
			3 & 1 & 8 & 8018\\
			4 & 7 & 9 & 62814\\
			5 & 37 & 10 & 556014 \\ \hline
  	\end{tabular} 
	\caption{Number of equations to be checked for $n$ state system}
	\label{table1}
\end{table*}
 Thus the number of equations that most be checked via Kolmogorov's method grows rapidly with $n$ (the number of states) and makes Kolmorogov's criterion computationally difficult for even moderate values of $n$. (See Online Encyclopedia of Integer Sequences, oeis A002807)

\section {Reversibility preserving matrix operations} 

 We next present a result to transform transition matrices in such a way as to preserve their reversibility status (either reversible or non-reversible). These transformations will be useful in creating new reversible Markov chains from existing ones, and for checking reversibility of Markov chains. 

We introduce a row multiplication operation on row $i$ of a Markov transition matrix as the 
multiplication of row $i$ by a positive constant that leaves the sum of the non diagonal elements at most 1, followed by an adjustment to $p_{ii}$ to  make the row sum exactly to 1.\\

We introduce a column multiplication operation on column $j$ of a Markov transition matrix as the 
multiplication of column $j$ by a positive constant of allowable size (so no row sums exceed 1) followed by adjustments to all diagonal entires to make every row sum exactly 1.\\

\begin{theorem} \label{thm1}
 A  Markov chain matrix $P$ maintains its reversibility status after a row multiplication operation or a column multiplication operation.    
\end{theorem} 
\begin{proof} 
Let the $i$th row of $P$ correspond to state $i$.  The Kolmogorov criterion states that $P$ is reversible iff for all loops of distinct states, the forward loop probability product equals the backward loop probability product.  If a loop does include state $i$, then a multiplication row operation on row $i$ has no effect on the forward and backward loop products. Otherwise, note that state $i$ appears in the first subscript of a forward loop probability iff it appears in the first subscript of a backward loop probability. So the row operation will have an identical effect on both sides of the loop product. A similar conclusion holds for column product operations. 
\end{proof}

\noindent
Note: If we let  $P^*$  be the matrix resulting from a row (or column) multiplication operation on $P$, then the limiting probabilities for $P^*$ in the above theorem are generally different than the limiting probabilities for $P$.\\

\noindent 
\begin{property}
 If the zeros in an ergodic Markov transition matrix are not symmetric, then the matrix is not reversible. 
 \end{property}
 \begin{proof}
 This follows by noting that the detailed balance equations fail for the nonsymmetric zero states, since the limiting probabilities all all entries of an ergodic chain are nonzero.  
 \end{proof}
 
\noindent
ALGORITHM: \\
If the matrix $P$ is an $n\times n$ probability transition matrix, then a sequence of at most $n-1$ row or column multiplication operations will be sufficient to determine whether or not $P$ is reversible or not. \\
(1) Pick two nonzero symmetric positions in P, say $p_{i1,i2}$ and $p_{i2,i1}$. Let $S_1=\{i1,i2\}$. If $p_{i1,i2}=p_{i2,i1}$, move to the next step. Otherwise, assume $p_{i1,i2}<p_{i2,i1}$. Multiply row $i2$ by $p_{i1,i2}/p_{i2,i1}$ and adjust $p_{i2,i2}$ to make the row sum to 1.
If $p_{i1,i2}>p_{i2,i1}$, multiply column $i2$ by $p_{i2,i1}/p_{i1,i2}$ and adjust all diagonal entries so that the rows sum to 1.\\ The new matrix $P^*$ will now have $p^*_{i1,i2}=p^*_{i2,i1}$.\\
(2) Choose another state $i3$ which has nonzero transition probabilities with a state in $S_1$. Make the appropriate row or column multiplication operation on row or column $i3$. 
Set $S_2=\{i1,i2,i3\}$.\\
(3) Repeat step 2 with a new state until there are no states left to add. After $n-1$ steps we have $S_n=\{i_1,\dots,i_n\}$. Let $P^*$ be the final matrix.

\begin{theorem} Let $P$ be an $n\times n$ transition matrix to which the Algorithm is applied, resulting in $P^{*}$. Then $P$ is reversible iff $P^{*}$ is symmetric. 
\end{theorem}
\begin{proof}
If $P^*$ is symmetric, then by earlier comments, $P^*$ is reversible. By Theorem 3.1, $P$ is reversible. \\
Next assume that $P$ is reversible. Then by Theorem 3.1, $P^*$ is reversible. Let $\vec{\tau}=(\tau_1,\dots, \tau_n)$ be the stationary vector for $P^*$. Note that $P^*$ was formed so that
$p^*_{ij}=p^*_{ji}$ for particular subcollection of $(i,j)$ which will include each of the states $1,...,n$ somewhere among the $(i,j)$ pairs  
So let $(i,j)$ be part of the subcollection. Since  $P^*$ is reversible, we must have detailed balance so  
  $\tau_i p^*_{ij} = \tau_j p^*_{ji}$, for that particular $i,j$.  But $p^*_{ij}=p^*_{ji}$ so  $\tau_i = \tau_j $ for the particular pair $(i,j)$. But all the states $1,2,\dots,n$ appear somewhere in the subcollection so we conclude $\tau_1=\tau_2=\dots =\tau_n$. Now we take an arbitrary pair $(i,j)$. Since $P^*$ is reversible, we have detailed balance for all $i,j$. Hence
$\tau_i p^*_{ij} = \tau_j p^*_{ji}$ for all $i,j$ and since $\tau_i=\tau_j$, it follows that $ p^*_{ij} =  p^*_{ji}$  for all $i,j$ so $P^*$ is symmetric.  
\end{proof}

\noindent
Notes
(1) We observe that if we conclude that a Markov transition matrix is reversible (using the Algorithm), then detailed balance will greatly simplify the computation of stationary vector $\vec{\pi}$. \\
(2) The Algorithm choose the smaller of two matrix entries in order modify the matrix $P$. Also corrections were made to the diagonal elements to ensure that the rows sum to 1. In fact, this is not really necessary and the matrix $p^*$ no longer needs to be a transition matrix. The important issue is whether $P^*$ is symmetric or not. \\
(3) We may be able to conclude that $P$ is not a reversible matrix by looking for symmetry only in the upper left corner of $P^*$ while it is being formed. That can save considerable computation.\\
(4) Although our results are stated for transition matrices for a finite state space, the same Algorithm could be used to check infinite state spaces if there were a patterned matrix (as in quasi birth and death processes).

\begin{flushleft}
\textbf{Example} 
\end{flushleft}
Let 
   \begin{center}       
      $P = \left[
         \begin{array}{c c c c} 
        .425 & .000 & .075 &.500\\
         .000 & .550 & .250 &.200 \\
         .300 & .250 & .450 &.000\\
         .500 &.050 &.000 &.450        
         \end{array}
      \right]$,
   \end{center}
To check for reversibility, we transform $P$ by column or row operations. Note that the zeros of P are symmetric (i.e. $0=p_{12}=p_{21}$ and $0=p_{34}=p_{43}$). We also note that $p_{14}=p_{41}$ so that symmetry for states $S=\{1,4\}$ already exists. We next try make the (1,3) entry match the (3,1) entry. We do not want to lose our (1,4) and (4,1)  symmetry, so we could either multiply column 3 by .300/.075 or we could multiply row 3 by .075/.300. We make the latter choice. We get a new matrix with row 3 equal to $(.075,.0625,.1125,.000)$. This new matrix is not a transition matrix because row 3 no longer sums to 1. We simply change the diagonal entry (3,3)  to $1-.075-.0625=.8625$. This yields the matrix \\
\begin{center}
   $P^{(1)}=   \left[
         \begin{array}{c c c c} 
        .425 & .000 & .075 &.500\\
         .000 & .550 & .250 &.200 \\
         .075 &.0625 &.8625 &.000\\
         .500 &.050 &.000 &.450        
         \end{array}
             \right]$
\end{center}

Then $P$ is reversible iff $P^{(1)}$ is reversible. We now have $S=\{1,4,3\}$. We next need to include state 2 in our computations. We need to preserve our values in entries (1,4) and (4,1), (1,3) and (3,1), so we can only change row 2 or column 2. We choose to multiply column 2 by 4 to make
entries (4,2) and (2,4) equal. Our result is\\
\begin{center}
   $P^{(2)}=   \left[
         \begin{array}{c c c c} 
        .425 & .000 & .075 &.500\\
         .000 & 2.2 & .250 &.200 \\
         .075 &.250 &.8625 &.000\\
         .500 &.200 &.000 &.450        
         \end{array}
             \right]$
\end{center}
Again, the rows do not sum to 1 so we change the diagonal entries of rows 2,3,4 to fix this. The result is \\
\begin{center}
   $P^{(3)}=   \left[
         \begin{array}{c c c c} 
        .425 & .000 & .075 &.500\\
         .000 & .550 & .250 &.200 \\
         .075 &.250 &.675 &.000\\
         .500 &.050 &.000 &.300        
         \end{array}
             \right]$
\end{center}

Now $P$ is reversible iff $P^{(3)}$ is reversible. But $P^{(3)}$ is symmetric so we know that it is automatically reversible. Thus, by Theorem 3.3,  $P$ is reversible. 
Since $P$ is reversible,  using detailed balance, we have
$\pi_1p_{14}=\pi_4p_{41}$ and $\pi_1p_{13}=\pi_3p_{31}$ and $\pi_3p_{32}=\pi_2p_{23}$ so $\pi_1(.5)=\pi_4(.5)$ and $\pi_1(.075)=\pi_3(.3)$ and $\pi_3(.25)=\pi_{2}(.25)$. Hence $\pi_2=\pi_3=3\pi_1=3\pi_4$. Since the sum of the probabilities is 1, we have $\pi_1=\pi_4=3/8$ and $\pi_2=\pi_3=1/8$. 

\section{Conclusion}
The procedure we present in this paper is based on Kolmogorov's loop criterion and  the detailed balance equations.  Our procedure to check for reversibility differs from the traditional approach in that we are using a transformed matrix $P^*$ obtained the initial transition matrix $P$. The advantage of the new approach is it only requires at most $k - 1$ elementary row or column operations to obtain the matrix $P^*$, and hence determine the reversibility status.  If we begin with an $n\times n$ transition  matrix with each entry equal to $1/n$ (and hence clearly symmetric), we can use the row and column operations to create a large variety of nonsymmetric transition matrices which are reversible.  
   
\section*{Acknowledgments}

This research was funded through a grant from NSERC (Natural
Sciences and Engineering Research Council of Canada).

\end{document}